\documentclass{amsart}

\usepackage{amssymb}
\usepackage[hidelinks]{hyperref}
\usepackage{todonotes}

\theoremstyle{plain}
\newtheorem{theorem}{Theorem}[section]
\newtheorem{lemma}[theorem]{Lemma}

\newtheorem{proposition}[theorem]{Proposition}
\newtheorem{conjecture}[theorem]{Conjecture}

\theoremstyle{definition}
\newtheorem{definition}[theorem]{Definition}

\theoremstyle{remark}

\newtheorem*{claim}{Claim}

\newcommand{\cI}{\mathcal{I}}
\newcommand{\I}{\cI}
\newcommand{\cJ}{\mathcal{J}}
\newcommand{\J}{\cJ}

\newcommand{\cP}{\mathcal{P}}

\DeclareMathOperator{\sq}{sq}
\DeclareMathOperator{\rk}{rk}

\newcommand{\fin}{\mathrm{Fin}}
\newcommand{\Fin}{\mathrm{Fin}}
\newcommand{\CEX}{\mathcal{CEI}}

\begin{document}

%%%%%%%%%%%%%%%%%%%%%%%%%%%%%%%%%%%%%%%%%%%%%%%%%%
%%%%% TITLE
%%%%%%%%%%%%%%%%%%%%%%%%%%%%%%%%%%%%%%%%%%%%%%%%%%

\title{On a conjecture of Debs and Saint Raymond}

%%%%%%%%%%%%%%%%%%%%%%%%%%%%%%%%%%%%%%%%%%%%%%%%%%
%%%%% AUTHORS
%%%%%%%%%%%%%%%%%%%%%%%%%%%%%%%%%%%%%%%%%%%%%%%%%%

\author[A.~Kwela]{Adam Kwela}
\address[Adam Kwela]{Institute of Mathematics\\ Faculty of Mathematics\\ Physics and Informatics\\ University of Gda\'{n}sk\\ ul.~Wita  Stwosza 57\\ 80-308 Gda\'{n}sk\\ Poland}
\email{Adam.Kwela@ug.edu.pl}
\urladdr{http://kwela.strony.ug.edu.pl/}

%%%%%%%%%%%%%%%%%%%%%%%%%%%%%%%%%%%%%%%%%%%%%%%%%%
%%%%% MSC (MATHEMATICAL SUBJECT CLASSIFICATION)
%%%%%%%%%%%%%%%%%%%%%%%%%%%%%%%%%%%%%%%%%%%%%%%%%%


%%%%%%%%%%%%%%%%%%%%%%%%%%%%%%%%%%%%%%%%%%%%%%%%%%
%%%%% KEYWORDS
%%%%%%%%%%%%%%%%%%%%%%%%%%%%%%%%%%%%%%%%%%%%%%%%%%


%%%%%%%%%%%%%%%%%%%%%%%%%%%%%%%%%%%%%%%%%%%%%%%%%%
%%%%% ABSTRACT
%%%%%%%%%%%%%%%%%%%%%%%%%%%%%%%%%%%%%%%%%%%%%%%%%%

\begin{abstract}
Borel separation rank of an analytic ideal $\mathcal{I}$ on $\omega$ is the minimal ordinal $\alpha<\omega_{1}$ such that there is $\mathcal{S}\in\bf{\Sigma^0_{1+\alpha}}$ with $\mathcal{I}\subseteq \mathcal{S}$ and $\mathcal{I}^\star\cap \mathcal{S}=\emptyset$, where $\mathcal{I}^\star$ is the filter dual to the ideal $\mathcal{I}$. Answering in negative a question of G. Debs and J. Saint Raymond [Fund. Math. 204 (2009), no. 3], we construct a Borel ideal of rank $>2$ which does not contain an isomorphic copy of the ideal $\text{Fin}^3$. 
\end{abstract}

%%%%%%%%%%%%%%%%%%%%%%%%%%%%%%%%%%%%%%%%%%%%%%%%%%
%%%%% MAKE TITLE
%%%%%%%%%%%%%%%%%%%%%%%%%%%%%%%%%%%%%%%%%%%%%%%%%%

\maketitle

%%%%%%%%%%%%%%%%%%%%%%%%%%%%%%%%%%%%%%%%%%%%%%%%%%
%%%%% SECTION
%%%%%%%%%%%%%%%%%%%%%%%%%%%%%%%%%%%%%%%%%%%%%%%%%%

\section{Introduction}

A collection $\mathcal{I}$ of subsets of a set $X$ is called an \emph{ideal on $X$} if it is closed under subsets and finite unions of its elements. We assume additionally that $\mathcal{P}(X)$ (i.e., the power set of $X$) is not an ideal, and that every ideal contains all finite subsets of $X$ (hence, $X=\bigcup\I$). All ideals considered in this paper are defined on infinite countable sets. 

We treat the power set $\mathcal{P}(X)$ as the space $2^X$ of all functions $f:X\rightarrow 2$ (equipped with the product topology, where each space $2=\left\{0,1\right\}$ carries the discrete topology) by identifying subsets of $X$ with their characteristic functions. Thus, we can talk about descriptive complexity of subsets of $\mathcal{P}(X)$ (in particular, of ideals on $X$). 

For $A,B,S\subseteq\cP(X)$ we say that \emph{$S$ separates $A$ from $B$} if $A\subseteq S$ and $S\cap B=\emptyset$. Following G. Debs and J. Saint Raymond \cite{Debs}, for an analytic ideal $\I$ we define its \emph{Borel separation rank} by:
$$\rk(\mathcal{I})=\min\left\{\alpha<\omega_{1}: \textrm{ there is }\mathcal{S}\in\bf{\Sigma^0_{1+\alpha}} \textrm{ separating }\mathcal{I}\textrm{ from }\mathcal{I}^\star\right\},$$
where $\mathcal{I}^\star=\{A^c:\ A\in\I\}$ is the \emph{filter dual to the ideal $\mathcal{I}$} (actually, authors of \cite{Debs} use the dual notion of filters instead of ideals). In this paper by \emph{rank of $\I$} we mean $\rk(\mathcal{I})$. 

This article is motivated by a conjecture of G. Debs and J. Saint Raymond from 2009 concerning combinatorial characterization of ideals of a given rank. Before formulating the mentioned conjecture, we need to introduce some tools.

Let $\I$ and $\J$ be ideals. Then:
\begin{itemize} 
\item $\I$ and $\J$ are \emph{isomorphic} if there is a bijection $f:\bigcup\J\to \bigcup\I$ such that: 
$$\forall_{A\subseteq\bigcup\I}\ \left(A\in\I\ \Longleftrightarrow\ f^{-1}[A]\in\J\right);$$
\item \emph{$\J$ contains an isomorphic copy of $\I$} ($\I\sqsubseteq\J$) if there is a bijection $f:\bigcup\J\to \bigcup\I$ such that: 
$$\forall_{A\subseteq\bigcup\I}\ \left(A\in\I\ \Longrightarrow\ f^{-1}[A]\in\J\right);$$
\item the \emph{Fubini product of $\I$ and $\J$} is an ideal given by:
$$\I\otimes\J=\left\{A\subseteq\bigcup\I\times\bigcup\J:\ \left\{x\in\bigcup\I:\ A_{(x)}\notin\J\right\}\in\I\right\},$$
where $A_{(x)}=\{y\in\bigcup\J:\ (x,y)\in A\}$.
\end{itemize}
Isomorphisms of ideals have been deeply studied for instance in \cite{Tryba} (see also \cite{Farah}), while the preorder $\sqsubseteq$ is examined in \cite{Kat}, \cite{Debs} and \cite{Reclaw}. 

Inspired by M. Kat\v{e}tov (cf. \cite[p. 240]{Katetov}), G. Debs and J. Saint Raymond defined ideals $\text{Fin}_\alpha$ such that $\rk(\fin_\alpha)=\alpha$, for all $0<\alpha<\omega_1$. In the case of finite ranks we have $\fin_1=\fin=[\omega]^{<\omega}$ and $\fin_{n+1}=\fin^{n+1}=\fin\otimes\fin^n$. The mentioned conjecture of G. Debs and J. Saint Raymond from 2009 is the following:

\begin{conjecture}[{\cite[Conjecture 7.8]{Debs}}]
\label{hip}
Let $\I$ be an analytic ideal. Then $\rk(\mathcal{I})\geq\alpha$ if and only if $\mathcal{I}$ contains an isomorphic copy of $\text{Fin}_\alpha$.
\end{conjecture} 

In \cite{IndLim} we have shown that the above is false for $\alpha=\omega$ and proposed new versions of $\fin_\alpha$ in the case of limit ordinals, for which the conjecture could remain true. Moreover, the implication $\Longleftarrow$ is true in general. In the next section we show that Conjecture \ref{hip} is false in the case of $\alpha=3$ by constructing a $\bf{\Sigma^0_6}$ ideal of rank $>2$ which does not contain an isomorphic copy of $\text{Fin}^3$. 

Our example is the best possible in the sense that Conjecture \ref{hip} is true for $\alpha=1$ and $\alpha=2$ as shown in \cite[Theorem 7.5]{Debs} and \cite[Theorem 4]{Reclaw}. What is interesting, although both results prove almost the same thing, they use completely different methods -- the methods of G. Debs and J. Saint Raymond from \cite{Debs} allowed to extend some of their results for the case of higher ranks and develop a bigger theory, while the method of M. Laczkovich and I. Rec\l aw from \cite{Reclaw} gives us more information about the separating set. The latter method has been used for instance in \cite{b} in a completely different context associated to some variants of the bounding number. 

It is also worth mentioning that a property of ideals can often be expressed by finding a critical ideal in sense of the Kat\v{e}tov preorder $\leq_K$ with respect to this property (recall that $\I\leq_K\J$, if there is $f:\bigcup\J\to \bigcup\I$, not necessary a bijection, such that $f^{-1}[A]\in\J$ for each $A\in\I$). This approach proved to be effective in many papers including \cite{Hrusak}, \cite{Hrusak2}, \cite{WR} and \cite{Meza}. By \cite[Example 4.1]{Kat}, $\fin^n\sqsubseteq\I$ if and only if $\fin^n\leq_K\I$, for any $n\in\omega$ and any ideal $\I$. Thus, Conjecture \ref{hip} can be seen as an attempt to characterize ideals of given rank in the above mentioned way.

Let us point out a connection of the above conjecture to descriptive complexity of ideals. Each ideal $\Fin_n$, for $n\in\omega$, is $\bf{\Sigma^0_{2n}}$. By a result of G. Debs and J. Saint Raymond from \cite{Debs2}, any ideal $\I$ containing an isomorphic copy of $\Fin_n$ cannot be $\Pi^0_{2n}$. Thus, Conjecture \ref{hip} would lead to an interesting result, by giving lower estimate of Borel complexity of ideals of a given finite rank. Although we show that Conjecture \ref{hip} is false, the general question about minimal complexity of an ideal of a given rank remains open (see \cite[Conjecture 8.5]{Debs}). It is only known that no $\Pi^0_{4}$ ideal can have rank $>1$ (cf. \cite[Theorem 9.1]{Debs}).

In \cite{Debs} and \cite{Reclaw} it is shown that ranks of analytic ideals are important for studying ideal pointwise limits of sequences of continuous functions: $\rk(\mathcal{I})=\alpha$ if and only if the family of all $\I$-pointwise limits of continuous real-valued functions defined on a given zero-dimensional Polish space is equal to the family of all functions of Borel class $\alpha$ (the definition of $\I$-pointwise limit can be found in \cite{Debs} or \cite{Reclaw}). Conjecture \ref{hip} would imply that we can skip the assumption about zero-dimensionality in this result (cf. \cite[Corollary 7.6]{Debs}). In fact, studies of $\I$-pointwise limits of continuous functions were the main motivation for introducing ranks of ideals (however, earlier S. Solecki in \cite{Solecki} studied this notion in a different context).

\section{The counterexample}

We start this section by introducing the ideal which will be a counterexample for Conjecture \ref{hip}.

\begin{definition}
\label{def}
The \emph{Counterexample ideal} $\CEX$ is a collection of subsets of $\omega^4$ consisting of all $A\subseteq\omega^4$ such that there is $n\in\omega$ satisfying:
\begin{itemize}
\item $\forall_{i<n}\ A_{(i)}\in\fin^3$;
\item $\forall_{i\geq n}\ \exists_{k\in\omega}\ \forall_{j\geq k}\ \exists_{m\in\omega}\ A\cap\left(\bigcup_{l\geq m}\{(i,j,l)\}\times\omega\right)\text{ is finite}. $
\end{itemize}
Equivalently, $\CEX=\left(\{\emptyset\}\otimes\fin^3\right)\cap\left(\fin^3\otimes\{\emptyset\}\right)$.
\end{definition}

\begin{proposition}
\label{1}
$\CEX$ is a Borel ideal of class $\bf{\Sigma^0_6}$. 
\end{proposition}

\begin{proof}
It is easy to verify that $\CEX$ is indeed an ideal.

The Borel complexity of $\CEX$ follows from the facts that $\fin^3$ is $\bf{\Sigma^0_6}$ (see \cite[Proposition 6.4]{Debs}) and that for every $(i,j,m)\in\omega^3$ the map $f_{i,j,m}:\cP(\omega^4)\to\cP(\omega^4)$ given by $f_{i,j,m}(A)=A\cap\left(\bigcup_{l\geq m}\{(i,j,l)\times\omega\}\right)$ is continuous, so the set:
$$\left\{A\subseteq\omega^4:\ A\cap\left(\bigcup_{l\geq m}\{(i,j,l)\times\omega\}\right)\text{ is finite}\right\}=f^{-1}_{i,j,l}[[\omega^4]^{<\omega}]$$
is $\bf{\Sigma^0_2}$.
\end{proof}

\begin{lemma}
\label{2}
$\fin^3\not\sqsubseteq\CEX$.
\end{lemma}

\begin{proof}
Assume to the contrary that there is a bijection $f:\omega^3\to\omega^4$ witnessing $\fin^3\sqsubseteq\CEX$. Denote:

\begin{equation*}
\begin{split}
T= & \left\{(i,j,k)\in\omega^3:\ (\{(i,j,k)\}\times\omega)\cap f[\{(n,m)\}\times\omega]\neq\emptyset\right.\\
& \left.\text{ for infinitely many }(n,m)\in\omega^2\right\}
\end{split}
\end{equation*}
and $Y=T\times\omega$.

\begin{claim}
$Y\in\CEX$.
\end{claim}

\begin{proof}
Assume to the contrary that $Y\notin\CEX$. Then $|T|=\omega$. Fix any bijections $g:\omega\to\omega^2$ and $h:\omega\to T$ and for each $n\in\omega$ pick a finite set $P_n\in[\omega^3]^{<\omega}$ such that:
\begin{itemize} 
\item $P_n\subseteq\{g(n)\}\times\omega$;
\item if $n'\leq n$ and $(\{h(n')\}\times\omega)\cap f[\{g(n)\}\times\omega]\neq\emptyset$ then $f[P_n]\cap(\{h(n')\}\times\omega)\neq\emptyset$.
\end{itemize}
Define $P=\bigcup_{n\in\omega}P_n$. Clearly, $P\in\fin^3$. However, $f[P]\notin\CEX$ as for every $(i,j,k)\in T$ the set $f[P]\cap(\{(i,j,k)\}\times\omega)$ is infinite and $Y=T\times\omega\notin\CEX$. This contradicts the choice of $f$.
\end{proof}

Define now:
\begin{equation*}
\begin{split}
S= & \left\{(i,j)\in\omega^2:\ \left(\{(i,j)\}\times\omega^2\right)\setminus Y\text{ is covered by finitely}\right.\\
& \left.\text{many sets of the form }f[\{(n,m)\}\times\omega]\right\}
\end{split}
\end{equation*}
and $Z=S\times\omega^2$. 

\begin{claim}
$Z\in\CEX$.
\end{claim}

\begin{proof}
Assume to the contrary that $Z\notin\CEX$. Then also $Z\setminus Y\notin\CEX$ (as $Y\in\CEX$). Consider the set: 
$$S'=\left\{(i,j,k)\in\omega^3:\ (i,j,k)\notin T\text{ and }(i,j)\in S\right\}.$$
Then $S'\times\omega=Z\setminus Y\notin\CEX$, so there has to be $i_0\in\omega$ such that $S'_{(i_0)}\notin\fin^2$ (as for each $A\subseteq\omega^3$ with $A_{(n)}\in\fin^2$ for every $n$ we have $A\times\omega\in\CEX$). Hence, $J=\{j\in\omega:\ (i_0,j,k)\in S'\text{ for infinitely many }k\}$ is infinite. Thus, for each $j\in J$ we can find $(n_j,m_j)\in\omega^2$ such that: 
$$\left\{k\in\omega: (i_0,j,k)\in S'\text{ and }|(\{(i_0,j,k)\}\times\omega)\cap f[\{(n_j,m_j)\}\times\omega]|=\omega\right\}$$
is infinite (as $(\{(i_0,j)\}\times\omega^2)\setminus Y$ is covered by finitely many sets of the form $f[\{(n,m)\}\times\omega]$).

Recall that the ideal $\fin^2$ is tall (see \cite[page 24]{Meza}), i.e., for each infinite $B\subseteq\omega^2$ there is an infinite $C\subseteq B$ with $C\in\fin^2$. Therefore, there is some $R\in\fin^2$ with $(n_j,m_j)\in R$ for infinitely many $j\in J$. Then $R\times\omega\in\fin^3$, but $(f[R\times\omega])_{(i_0)}\notin\fin^3$ and consequently $f[R\times\omega]\notin\CEX$. Again we obtain a contradiction with the choice of $f$.
\end{proof}

Since $Z\in\CEX$, the set $\omega^2\setminus S$ is infinite. Let $g:\omega\to\omega^2$ and $h:\omega\to \omega^2\setminus S$ be fixed bijections. For every $n\in\omega$ pick a finite set $G_n\in[\omega^3]^{<\omega}$ such that:
\begin{itemize} 
\item $G_n\subseteq\{g(n)\}\times\omega$;
\item if $n'\leq n$ and $((\{h(n')\}\times\omega^2)\setminus Y)\cap f[\{g(n)\}\times\omega]\neq\emptyset$ then $f[G_n]\cap((\{h(n')\}\times\omega^2)\setminus Y)\neq\emptyset$.
\end{itemize}
Define $G=\bigcup_{n\in\omega}G_n$. Obviously, $G\in\fin^3$. However, $f[G]\notin\CEX$ as $(S\times\omega^2)\cup Y\in\CEX$ and for every $(i,j)\notin S$ and $k\in\omega$ we have: 
$$\left|f[G]\cap\left(\bigcup_{l\geq k}\left(\{(i,j,l)\}\times\omega\right)\setminus Y\right)\right|=\omega$$ 
(since $\bigcup_{l<k}(\{(i,j,l)\}\times\omega)\setminus Y$ is covered by finitely many sets of the form $f[\{(n,m)\}\times\omega]$ while $(\{(i,j)\}\times\omega^2)\setminus Y$ is not). This contradicts the choice of $f$ and finishes the entire proof.
\end{proof}

\begin{lemma}
\label{3}
$\rk(\CEX)>2$.
\end{lemma}

\begin{proof}
Suppose that $S\in\bf{\Sigma^0_3}$ is arbitrary such that $\CEX\subseteq S$. Then there are finite sets $F_{i,j,k}, G_{i,j,k}\subseteq\omega^4$, for $i,j,k\in\omega$, such that:
$$S=\left\{A\subseteq\omega^4:\ \exists_{i\in\omega}\ \forall_{j\in\omega}\ \exists_{k\in\omega}\ F_{i,j,k}\cap A=\emptyset\text{ and }G_{i,j,k}\subseteq A\right\}.$$

We will construct a set $X\in\CEX$ such that: 
$$\exists_{i\in\omega}\ \forall_{j\in\omega}\ \exists_{k\in\omega}\ F_{i,j,k}\subseteq X\text{ and }G_{i,j,k}\cap X=\emptyset.$$
This will finish the proof, since the above implies $\omega^4\setminus X\in\CEX^\star\cap S$ and consequently $S\cap\CEX^\star\neq\emptyset$.

We need to introduce some notation. For each $A\in S$ there is $I(A)\subseteq\omega$, $I(A)\neq\emptyset$ such that for each $i\in I(A)$ we have: 
$$\forall_{j\in\omega}\ \exists_{k\in\omega}\ F_{i,j,k}\cap A=\emptyset\text{ and }G_{i,j,k}\subseteq A.$$
Moreover, for each $A\in S$, $i\in I(A)$ and $j\in\omega$ let $k(A,i,j)\in\omega$ be such that $F_{i,j,k(A,i,j)}\cap A=\emptyset$ and $G_{i,j,k(A,i,j)}\subseteq A$.

For $s=(s_0,s_1,\ldots)\in\omega^{\omega}$ denote:
$$A(s)=\bigcup_{m\in\omega}\left(\left((\omega\setminus m)\times s_m\times\omega^2\right)\cup\left(\omega\times(\omega\setminus m)\times s_m\times\omega\right)\right)$$
(here and in the rest of this proof we use standard set-theoretic notation and identify $n\in\omega$ with the set $\{0,1,\ldots,n-1\}$). Observe that $A(s)\in\CEX$. Moreover, if $s,t\in\omega^{\omega}$ are such that $s_n\leq t_n$ for all $n$ then $A(s)\subseteq A(t)$ (we will use this observation without any reference). Put also $A((s_0,\ldots,s_m))=A((s_0,\ldots,s_m,0,0,\ldots))$ for each $s=(s_0,\ldots,s_m)\in\omega^{<\omega}$. For $i'\in\omega$ denote by $0_{i'}\in\omega^{i'}$ the sequence consisting of $i'$ zeros and define $\sq(i')=i'\times i'\times\omega^2$. Finally, for each $A\in S$, $i\in I(A)$ and $i',j\in\omega$ put:
$$p(A,i,i',j)=\min\left\{p\in\omega:\ G_{i,j,k(A,i,j)}\subseteq \sq(i')\cup A(0_{i'}^\frown(p))\right\}.$$
Note that the above is well defined as the set $G_{i,j,k(A,i,j)}$ is finite and:
$$\bigcup_{p\in\omega}\left(\sq(i')\cup A(0_{i'}^\frown(p))\right)=\omega^4.$$

\begin{claim}
There are $i,i'\in\omega$, $T\subseteq \sq(i')$, $t\in\omega^{i'}$, $B\in[\omega^4]^{<\omega}$ and an infinite sequence $(C_m)\subseteq\CEX$ such that for each $m\in\omega$ we have:
\begin{itemize}
\item[(a)] $i\in I(C_m)$;
\item[(b)] $C_m\cap \sq(i')=T$;
\item[(c)] $C_m\supseteq A(t^\frown(p_m))\setminus\left(B\cup\bigcup_{j<m}F_{i,j,k(C_j,i,j)}\right)$, where: 
$$p_m=\max(\{p(C_j,i,i',j):\ j<m\}\cup\{m\});$$
\item[(d)] $C_m\cap\left(B\cup\bigcup_{j<m}F_{i,j,k(C_j,i,j)}\right)=\emptyset$.
\end{itemize}
\end{claim}

\begin{proof}
Suppose otherwise towards contradiction. Then for each $i,i'\in\omega$, $T\subseteq \sq(i')$, $t\in\omega^{i'}$ and $B\in[\omega^4]^{<\omega}$ either there is no $C_0\in\CEX$ such that $i\in I(C_0)$, $C_0\cap\sq(i')=T$, $C_0\supseteq A(t)\setminus B$ and $C_0\cap B=\emptyset$ or there are $C_0,C_1,\ldots,C_n\in\CEX$ such that for each $m\leq n$ we have:
\begin{itemize}
\item $i\in I(C_m)$;
\item $C_m\cap \sq(i')=T$;
\item $C_m\supseteq A(t^\frown(p_m))\setminus\left(B\cup\bigcup_{j<m}F_{i,j,k(C_j,i,j)}\right)$, where: 
$$p_m=\max(\{p(C_j,i,i',j):\ j<m\}\cup\{m\});$$
\item $C_m\cap\left(B\cup\bigcup_{j<m}F_{i,j,k(C_j,i,j)}\right)=\emptyset$;
\end{itemize}
but there is no $C_{n+1}\in\CEX$ such that:
\begin{itemize}
\item $i\in I(C_{n+1})$;
\item $C_{n+1}\cap \sq(i')=T$;
\item $C_{n+1}\supseteq A(t^\frown(p_{n+1}))\setminus\left(B\cup\bigcup_{j\leq n}F_{i,j,k(C_j,i,j)}\right)$, where:
$$p_{n+1}=\max(\{p(C_j,i,i',j):\ j\leq n\}\cup\{n+1\});$$
\item $C_{n+1}\cap\left(B\cup\bigcup_{j\leq n}F_{i,j,k(C_j,i,j)}\right)=\emptyset$.
\end{itemize}

For each $n\in\omega$ we will inductively pick $i_n\in\omega$, $b_n\in\omega\setminus\{0\}$, $(A^m_n)_{m<b_n}\subseteq\CEX$, $B_n\in[\omega^4]^{<\omega}$, $s_n\in\omega$ and $A_n\in\CEX$ such that:
\begin{itemize}
\item[(a)] $i_n$ is minimal such that there is $A\in\CEX$ with:
\begin{itemize}
\item $i_n\in I(A)$;
\item $A\cap\sq(n)=A_{n-1}\cap\sq(n)$;
\item $A\supseteq A_{n-1}$;
\item $A\cap B_{n-1}=\emptyset$;
\end{itemize}
\item[(b)] for each $m<b_n$ we have:
\begin{itemize}
\item $i_n\in I(A^m_n)$;
\item $A^m_n\cap\sq(n)=A_{n-1}\cap\sq(n)$;
\item $A^m_n\supseteq A((s_0,\ldots,s_{n-1},p^m_n))\setminus \left(B_{n-1}\cup\bigcup_{j<m}F_{i_n,j,k(A^j_n,i_n,j)}\right)$ (here we put $p^m_n=\max(\{p(A^j_n,i_n,n,j):\ j<m\}\cup\{m\})$);
\item $A^m_n\cap\left(B_{n-1}\cup\bigcup_{j<m}F_{i_n,j,k(A^j_n,i_n,j)}\right)=\emptyset$;
\end{itemize}
\item[(c)] $B_n=B_{n-1}\cup\bigcup_{j<b_n}F_{i_n,j,k(A^j_n,i_n,j)}$; 
\item[(d)] $s_n=\max(\{p(A^j_n,i_n,n,j):\ j<b_n\}\cup\{b_n\})$;
\item[(e)] $A_n=A((s_0,\ldots,s_{n-1},s_n))\setminus B_n$;
\item[(f)] there is no $A\in\CEX$ such that: 
\begin{itemize}
\item $i_n\in I(A)$;
\item $A\cap \sq(n)=A_{n-1}\cap \sq(n)$;
\item $A\supseteq A_n$;
\item $A\cap B_n=\emptyset$;
\end{itemize} 
\end{itemize} 
(in the above we put $A_{-1}=B_{-1}=\emptyset$).

In the first induction step let $i_0\in\omega$ be minimal such that there is some $A\in\CEX$ with $i_0\in I(A)$ (observe that $i_0$ satisfies item (a) as $\sq(0)=\emptyset$). Applying our assumption to $i=i_0$, $i'=0$, $T=\emptyset$, $t=\emptyset$ and $B=\emptyset$, we get $b_0\in\omega\setminus\{0\}$ and the required in item (b) sequence $(A^m_0)_{m<b_0}\subseteq\CEX$ (since we have $A\in\CEX$ such that $i\in I(A)$, $A\cap\sq(i')=\emptyset=T$, $A\supseteq \emptyset=A(t)\setminus B$ and $A\cap B=\emptyset$). Define $B_0$, $s_0$ and $A_0$ according to items (c), (d) and (e) and observe that item (f) is satisfied by the choice of $(A^m_0)_{m<b_0}$.

In the $n$th induction step, if $i_r$, $b_r$, $s_r$, $(A^m_r)_{m<b_r}$, $B_r$ and $A_r$, for all $r<n$, are already defined, let $i_n\in\omega$ be as in item (a). Note that such $i_n$ exists as $A=A_{n-1}\in\CEX\subseteq S$ is a set satisfying $A\cap\sq(n)=A_{n-1}\cap\sq(n)$, $A\supseteq A_{n-1}$ and $A\cap B_{n-1}=\emptyset$ (hence, $i_n\leq\min I(A_{n-1})$). Using our assumption applied to $i=i_n$, $i'=n$, $t=(s_0,\ldots,s_{n-1})$, $T=A_{n-1}\cap \sq(n)$ and $B=B_{n-1}$ once again we get $b_n\in\omega$ and a finite sequence $(A^m_n)_{m<b_n}\subseteq\CEX$ with the property required in (b). Put $B_n$, $s_n$ and $A_n$ according to items (c), (d) and (e). Again item (f) is satisfied by the choice of $(A^m_n)_{m<b_n}$. This finishes the inductive construction.

Notice that the sequence $(B_n)$ is non-decreasing. The rest of this proof strongly relies on the observation that $A_k=A(s_0,\ldots,s_k)\setminus\bigcup_{n\in\omega}B_n$, for every $k\in\omega$. Indeed, $A_k\supseteq A(s_0,\ldots,s_k)\setminus\bigcup_{n\in\omega}B_n$ is obvious, so we only need to prove $A_k\subseteq A(s_0,\ldots,s_k)\setminus\bigcup_{n\in\omega}B_n$. Fix $x\in A_k=A(s_0,\ldots,s_k)\setminus B_k$ and suppose to the contrary that $x\notin A(s_0,\ldots,s_k)\setminus\bigcup_{n\in\omega}B_n$, i.e., $x\in B_{n}\setminus B_{n-1}\subseteq\bigcup_{j<b_n}F_{i_n,j,k(A^j_n,i_n,j)}$ for some $n>k$. Let $m<b_{n}$ be minimal such that $x\in F_{i_n,m,k(A^m_n,i_n,m)}$. Then:
\begin{equation*}
\begin{split}
x\in &  A((s_0,\ldots,s_{k}))\setminus \left(B_{n-1}\cup\bigcup_{j<m}F_{i_n,j,k(A^j_n,i_n,j)}\right)\subseteq \\
& A((s_0,\ldots,s_{n-1}))\setminus \left(B_{n-1}\cup\bigcup_{j<m}F_{i_n,j,k(A^j_n,i_n,j)}\right)\subseteq A^m_n,
\end{split}
\end{equation*}
which contradicts $F_{i_n,m,k(A^m_n,i_n,m)}\cap A^m_n=\emptyset$.

Observe that the sequence $(i_n)$ is unbounded. Indeed, we will show that given any $n<n'$ we have $i_n\neq i_{n'}$ (from which unboundedness of $(i_n)$ easily follows). Fix any pair $n<n'$. Note that
$A\cap\sq(n')=A_{n'-1}\cap\sq(n')$ implies:
\begin{equation*}
\begin{split}
A\cap\sq(n)= &  A_{n'-1}\cap\sq(n)=\left(A(s_0,\ldots,s_{n'-1})\setminus\bigcup_{n\in\omega}B_n\right)\cap\sq(n)= \\
& \left(A(s_0,\ldots,s_{n-1})\setminus\bigcup_{n\in\omega}B_n\right)\cap\sq(n)=A_{n-1}\cap\sq(n).
\end{split}
\end{equation*}
Similarly, $A\supseteq A_{n'-1}=A(s_0,\ldots,s_{n'-1})\setminus\bigcup_{n\in\omega}B_n$ implies the inclusion $A\supseteq A(s_0,\ldots,s_{n})\setminus\bigcup_{n\in\omega}B_n=A_{n}$ and $A\cap B_{n'-1}=\emptyset$ implies that $A\cap B_{n}=\emptyset$. Hence, by item (f) there is no $A\in\CEX$ such that $i_n\in I(A)$, $A\cap\sq(n')=A_{n'-1}\cap\sq(n')$, $A\supseteq A_{n'-1}$ and $A\cap B_{n'-1}=\emptyset$. On the other hand, item (a) says that $i_{n'}$ is minimal such that such $A$ exists. Thus, $i_n$ and $i_{n'}$ cannot be the same.

Consider the set $D=A(s_0,s_1,\ldots)\setminus\bigcup_{n\in\omega}B_n\in\CEX$. We will show that $I(D)=\emptyset$, which will contradict $\CEX\subseteq S$ and finish the proof. 

Fix any $j\in\omega$ and let $k\in\omega$ be such that $j<i_{k}$ (such $k$ exists as $(i_n)$ is unbounded). Clearly, $D\cap B_{k-1}=\emptyset$. Moreover, $D\supseteq A_{k-1}$, as:
$$D=A(s_0,s_1,\ldots)\setminus\bigcup_{n\in\omega}B_n\supseteq A(s_0,\ldots,s_{k-1})\setminus\bigcup_{n\in\omega}B_n=A_{k-1}.$$
Finally, $D\cap \sq(k)=A_{k-1}\cap \sq(k)$ since: 
\begin{equation*}
\begin{split}
D\cap \sq(k)= &  \left(A(s_0,s_1,\ldots)\setminus\bigcup_{n\in\omega}B_n\right)\cap \sq(k)= \\
& \left(A(s_0,\ldots,s_{k-1})\setminus\bigcup_{n\in\omega}B_n\right)\cap \sq(k)=A_{k-1}\cap \sq(k).
\end{split}
\end{equation*}
Thus, $j\notin I(D)$ by the definition of $i_k$ (item (a)). Since $j$ was arbitrary, we obtain $I(D)=\emptyset$. This finishes the claim.
\end{proof}

Let $i,i'\in\omega$, $T\subseteq \sq(i')$, $t\in\omega^{i'}$, $B\in[\omega^4]^{<\omega}$ and $(C_m)\subseteq\CEX$ be as in the above Claim. Define $X=\bigcup_{m\in\omega}F_{i,m,k(C_m,i,m)}$. Observe that $X\cap G_{i,m,k(C_m,i,m)}=\emptyset$ for all $m$. Indeed, fix $x\in X$ and suppose towards contradiction that $x\in G_{i,n,k(C_n,i,n)}$ for some $n\in\omega$. Denote by $m$ the minimal integer such that $x\in F_{i,m,k(C_{m},i,m)}$. There are three possibilities:
\begin{itemize}
\item The case $m=n$ cannot happen as we have $F_{i,n,k(C_{n},i,n)}\cap C_n=\emptyset$ and $G_{i,n,k(C_{n},i,n)}\subseteq C_n$.
\item The case $m<n$ is impossible since $F_{i,m,k(C_{m},i,m)}\cap C_n=\emptyset$ (by item (d)), while $G_{i,n,k(C_n,i,n)}\subseteq C_n$.
\item If $m>n$ then note that (b) and the definition of $p(C_n,i,i',n)$ imply:
\begin{equation*}
\begin{split}
x\in\  &  G_{i,n,k(C_n,i,n)}\subseteq C_n\cap\left(\sq(i')\cup A(0_{i'}^\frown(p(C_n,i,i',n))) \right)\subseteq  \\
& T\cup A(0_{i'}^\frown(p(C_n,i,i',n))).
\end{split}
\end{equation*}
Moreover, $x\notin B\cup\bigcup_{j<m}F_{i,j,k(C_j,i,j)}$ by $x\in G_{i,n,k(C_n,i,n)}\subseteq C_n$, $C_n\cap B=\emptyset$ (by (d)) and the choice of $m$. This in turn means that:
\begin{equation*}
\begin{split}
x\in\  &  \left(T\cup A(0_{i'}^\frown(p(C_n,i,i',n)))\right)\setminus \left(B\cup\bigcup_{j<m}F_{i,j,k(C_j,i,j)}\right)\subseteq \\
& T\cup \left(A(0_{i'}^\frown(p(C_n,i,i',n)))\setminus\left(B\cup\bigcup_{j<m}F_{i,j,k(C_j,i,j)}\right)\right)\subseteq C_{m}
\end{split}
\end{equation*}
(by (b) and (c)). The latter contradicts $F_{i,m,k(C_{m},i,m)}\cap C_{m}=\emptyset$. Hence, $m>n$ neither can hold.
\end{itemize}

Therefore, $i\in\omega$ is such that for each $m\in\omega$ there is $k=k(C_m,i,m)$ with $F_{i,m,k}\subseteq X$ and $G_{i,m,k}\cap X=\emptyset$. To finish the proof we need to show that $X\in\CEX$. 

We will show that $X\cap A(0_{i'}^\frown(m))\subseteq B\cup\bigcup_{j<m}F_{i,j,k(C_j,i,j)}$, for each $m$. Assume to the contrary that for some $m\in\omega$ there is $x\in X\cap A(0_{i'}^\frown(m))$ such that $x\notin B\cup\bigcup_{j<m}F_{i,j,k(C_j,i,j)}$ and let $m'\geq m$ be minimal such that $x\in F_{i,m',k(C_{m'},i,m')}$. Then:
$$x\in A(0_{i'}^\frown(m'))\setminus\left(B\cup\bigcup_{j<m'}F_{i,j,k(C_j,i,j)}\right)\subseteq C_{m'}.$$
(by item (c) and the fact that $m\leq m'$ implies $A(0_i^\frown(m))\subseteq A(0_i^\frown(m'))$). On the other hand, $F_{i,m',k(C_{m'},i,m')}\cap C_{m'}=\emptyset$. This contradiction proves that $X\cap A(0_{i'}^\frown(m))\subseteq B\cup\bigcup_{j<m}F_{i,j,k(C_j,i,j)}$.

Since $B\cup\bigcup_{j<m}F_{i,j,k(C_j,i,j)}$ is finite, for each $m$, the rest of the proof follows from the previous paragraph and an easy observation that each set $Y\subseteq\omega^4$ with $Y\cap A(0_{i'}^\frown(m))$ finite for all $m$, belongs to the ideal $\CEX$. Indeed, for each such set $Y$ there is $n=i'$ such as in Definition \ref{def}, i.e.,:
\begin{itemize}
\item if $i''\geq i'=n$ and $j\in\omega$ then: 
$$Y\cap\left(\{(i'',j)\}\times\omega^2\right)\subseteq Y\cap A(0_{i'}^\frown(j+1))\in[\omega^4]^{<\omega};$$
\item if $i''<i'=n$, $j\geq i'$ and $k\in\omega$ then:
$$Y\cap\left(\{(i'',j,k)\}\times\omega\right)\subseteq Y\cap A(0_{i'}^\frown(k+1))\in[\omega^4]^{<\omega}.$$
\end{itemize}
\end{proof}

\begin{theorem}
There is a $\bf{\Sigma^0_6}$ ideal of rank $>2$ not containing an isomorphic copy of $\fin^3$.
\end{theorem}

\begin{proof}
$\CEX$ is a good example as shown in Proposition \ref{1} and Lemmas \ref{2} and \ref{3}.
\end{proof}

%%%%%%%%%%%%%%%%%%%%%%%%%%%%%%%%%%%%%%%%%%%%%%%%%%
%%%%% REFERENCES
%%%%%%%%%%%%%%%%%%%%%%%%%%%%%%%%%%%%%%%%%%%%%%%%%%

\bibliographystyle{amsplain}
\bibliography{Ranks-references}

\end{document}